\documentclass[a4paper, draft, 11pt]{article}
\usepackage[margin=1.25in]{geometry}
\usepackage{authblk}
\usepackage{enumerate}
\usepackage{mathtools}
\usepackage{amsmath}
\usepackage{amsfonts}
\usepackage{amssymb}
\usepackage{enumitem}
\usepackage{fancyhdr}
\usepackage{parskip}
\usepackage{amsthm}
\usepackage{setspace}

{
  \theoremstyle{plain}
  
}

\begin{document}
\title{
A Probabilistic Proof of the Perron-Frobenius Theorem
}

\author[1]{Peter W. Glynn}
\author[2]{Paritosh Y. Desai}
\affil[1]{Department of Management Science and Engineering, Stanford University, Stanford CA}
\affil[2]{Target Corporation, Sunnyvale, CA}

\date{}
\maketitle

\abstract{The Perron-Frobenius theorem plays an important role in many areas of management science and operations research. This paper provides a probabilistic perspective on the theorem, by discussing a proof that exploits a probabilistic representation of the Perron-Frobenius eigenvalue and eigenvectors in terms of the dynamics of a Markov chain. The proof provides conditions in both the finite-dimensional and infinite-dimensional settings under which the Perron-Frobenius eigenvalue and eigenvectors exist. Furthermore, the probabilistic representations that arise can be used to produce a Monte Carlo algorithm for computing the Perron-Frobenius eigenvalue and eigenvectors that will be explored elsewhere.}

\section{Introduction}
\label{sec:intro}
The Perron-Frobenius theorem for non-negative matrices plays an important role in many areas of operations research, the management sciences, and applied mathematics in general. An important such setting is that of Markov population decision chains, in which Perron-Frobenius theory can be used to describe the rate at which the total reward aggregated across the entire population grows as a function of time; see, for example, \cite{RV75}, \cite{Pliska76}, \cite{Rothblum84},  \\ \cite{RW82}, and  \cite{RV92}. 

Given the importance of the Perron-Frobenius theorem, it is convenient to have multiple perspectives on the result. In this paper, we explore the Perron-Frobenius theorem from a probabilistic viewpoint. In particular, we give a proof of the Perron-Frobenius theorem that relies on a representation of the Perron-Frobenius eigenvalue and its associated row and column eigenvectors in terms of expectations defined via the dynamics of a Markov chain. The representation takes advantage of regenerative structure that is often available in the Markov setting. Assuming a simple condition (see A1 in Section 3), we can quickly argue the existence of a positive eigenvalue and corresponding positive row and column eigenvectors using our probabilistic methods; see the first part of the proof of Theorem 1.

But this probabilistic representation carries additional advantages. In particular, our methods extend directly to the setting of infinite-dimensional non-negative matrices and even to continuous state space non-negative operators. Such generalizations are useful in a number of applied probability settings, most particularly when analyzing the quasi-stationary behavior of a Markov chain living on a countable or continuous state space. In a follow-on paper, we will show how our probabilistic representation also allows one to develop Monte Carlo algorithms for computing the Perron-Frobenius eigenvalue and eigenvectors efficiently and without state space discretization, even in the setting of continuous state space. While some elements of the ideas we present here exist within the probability literature (see, for example, Section 2 of \cite{NN87}), our discussion is intended to provide a more unified and accessible framework within which the key probability aspects can be easily understood and identified.

      This paper is organized as follows. In Section 2, we study a particular infinite-dimensional non-negative matrix in some depth, to help clarify some of the complications that arise in extending the Perron-Frobenius theorem to the infinite-dimensional setting. In particular, the example shows that in such an infinite-dimensional context, a continuum of positive eigenvalues, each accompanied by positive row and column eigenvectors, can exist. However, only one such eigenvalue describes the growth rate of successive powers of entries of the matrix, and it is this eigenvalue that we view as the natural generalization of the finite-dimensional Perron-Frobenius eigenvalue. Thus, the Perron-Frobenius eigenvalue and eigenvectors are not characterized via positivity considerations in this infinite-dimensional setting. In Section 3, we focus on finite and infinite-dimensional matrices, and provide our probabilistic representation in Theorem 1. We then show that it corresponds to the Perron-Frobenius eigenvalue in Theorem 2. The rest of the section is devoted to a discussion of sufficient conditions guaranteeing the validity of the hypotheses underlying Theorems 1 and 2. Finally, in Section 4, we quickly show how the ideas naturally extend to continuous state space non-negative operators.

\section{An Example}
\label{sec:2ex}

We wish to provide a version of the Perron-Frobenius theorem that is  valid for infinite-dimensional  matrices, and that can be readily extended to continuous state space. Before stating and proving this result in the next section, we wish to provide an example that illustrates the additional complications that can arise in the infinite-dimensional setting.

Consider a birth-death Markov chain  $X = ( X_{n} : n \geq 0)$ having state space $\mathbf{N}_{+}=  \{0,1,2, \hdots\}$, with transition probabilities
\begin{align*}
  P(x, x+1)  =  & p &&, x \geq 0 \\
  P(x, x-1)  =  & q \hspace{1em} (\triangleq 1-p) &&, x \geq 1, 
\end{align*}
with $P(0, 0) = q $ (with $0 < p < 1)$. Such a Markov chain arises, for example, as a slotted-time queue in which the number of arrivals in each slot of time is either $2$ (with probability $p$) or $0$ (with probability $q$), and that can serve one customer per slot of time whenever a customer is available to be served. If we wish to study the dynamics of $X$ conditional on the system  not having emptied by time $n$, we are led to the consideration of probabilities of the form $P_{x} (X_{n} = y, T > n)$, where $P_{x}(.) \triangleq P(\cdot | X_{0} = x)$ and $T = \inf{} \{n \geq 0: X_{n} = 0 \}$. In other words, we are computing the quasi-stationary behavior of $X$, associated with non-entry into $\{0\}$ over the time interval $[0, n]$. Clearly,
\[
  P_{x}(X_{n} = y, T > n) = B^{n}(x, y),
\]
where $B = (B(x,y): x, y \in S)$ is the non-negative infinite-dimensional matrix in which $B(x,y) = P(x,y)$ for $x, y \in S = \{ 1, 2, \hdots \}$. The matrix $B$ is {\em irreducible}, in the sense that for each $x, y \in S$, there exists $n = n(x, y)$ for which $B^{n}(x, y) > 0$ and is sub-stochastic (since the row sum for state $1$ is $p < 1$).

The natural extension of the notion of a Perron-Frobenius eigenvalue and (column) eigenvector from the finite dimensional setting to the infinite-dimensional matrix $B$ involves consideration of the eigensystem
\begin{align}
  \label{eqn:eigenrel}
  Bu & = \lambda u
\end{align}
for positive solutions $u$ and $\lambda$. For this example, this corresponds to studying the linear system
\begin{align}
  \label{eqn:linsystem}
  p u(x+1) + q u(x-1) & = \lambda u(x), x \geq 2,  
\end{align}
with the boundary condition
\begin{align}
  \label{eqn:linsystem2}
  p u(2)  & = \lambda u(1).  
\end{align}
To solve the recursive equation (\ref{eqn:linsystem}), it is standard practice to consider solutions that are linear combinations of functions of the form $z^{x}$, thereby leading to consideration of a quadratic equation  for $z$, namely,
\[
  p z^{2} - \lambda z + q = 0.
\]
The roots of this equation are given by
\[
  z_{1}  = \frac{\lambda - \sqrt{\lambda^{2} - 4pq}}{2p}, 
\]
and
\begin{align}
  \label{eq:quadsol}
  z_{2} & = \frac{\lambda + \sqrt{\lambda^{2} - 4pq}}{2p}, 
\end{align}
provided $\lambda \geq 2\sqrt{pq}$. If $0 \leq \lambda < 2 \sqrt{pq}$, then $z_{1}$ and $z_{2}$ are complex, and no linear combination of $z_{1}^{x}$ and  $z_{2}^{x}$ can be non-negative for all $x \geq 1$, so that positive eigenvectors cannot exist for $\lambda$ in this range. We therefore  now restrict ourselves to  consideration of the case  in which $\lambda \geq 2 \sqrt{pq}$.

Suppose first that $\lambda > 2 \sqrt{pq}$. If we take $u(x) = a z_{1}^{x} +  z_{2}^{x} = z_{1}^{x} (a + (z_{2}/z_{1})^{x})$, then equation (\ref{eqn:linsystem2}) requires that
\[
  p( a z_{1}^{2} + z_{2}^{2}) = \lambda (a z_{1} + z_{2}),
\]
in which case
\[
  a = - \frac{z_{2}}{z_{1}} \frac{\lambda - p z_{2}}{\lambda - p z_{1}}.
\]

Since $0 < z_{1} < z_{2} < \frac{\lambda}{p}$, $a < 0$. However,
\begin{align*}
  a + (\frac{z_{2}}{z_{1}})^{1} & = \frac{z_{2}}{z_{1}} ( 1 - \frac{\lambda - p z_{2}}{\lambda - p z_{1}}) \\
  & = p \frac{z_{2}}{z_{1}} \frac{z_{2}-z_{1}}{\lambda - p z_{1}}  > 0, \\
\end{align*}
so $a + (\frac{z_{2}}{z_{1}})^{x} > 0$ for $x \geq 1$, and hence $u = (u(x): x \geq 1)$ is a strictly positive column eigenvector solution to equation (\ref{eqn:eigenrel}) for every $\lambda > 2 \sqrt{pq}$. So, there is a continuum of positive $\lambda$'s for which positive eigenvector solutions to equation (\ref{eqn:eigenrel}) exist.

For $\lambda = 2 \sqrt{pq}$, $z_{1} = z_{2} = \sqrt{\frac{q}{p}}$ and we are led to consideration of a solution  to equation (\ref{eqn:linsystem}) of the form
\[
  u(x) = ({q}/{p})^\frac{x}{2}(a + x)
\]
for some choice of $a$. In this case, equation (\ref{eqn:linsystem2}) requires that $a=0$, so the solution $u(x) = x (q/p)^{\frac{x}{2}}$ for $x \geq 1$ is a positive column eigenvector solution of equation (\ref{eqn:eigenrel}) when $\lambda = 2 \sqrt{pq}$.

We can similarly consider the row eigenvector $\eta = (\eta(x):  x \in S)$ corresponding to eigenvalue $\lambda$, namely
\begin{align}
  \label{eqn:etaeigen}
  \eta B = \lambda \eta.
\end{align}
Again, we are interested in positive solutions $\lambda$ and $\eta$ to equation (\ref{eqn:etaeigen}). Here, the associated linear system is
\begin{align}
  \label{eqn:linsystem3}
  p \eta(y-1) + q \eta(y+1) = \lambda \eta(y), y \geq 2,
\end{align}
subject to the boundary condition
\begin{align}
  \label{eqn:linsystem4}
  q \eta(2) = \lambda \eta(1).
\end{align}

As in the column vector setting, $\eta$ can be computed explicitly. The general solution takes the form
\[
  \eta (x) = \tilde{z_{1}}^{x} ( b + (\frac{\tilde{z_{2}}}{\tilde{z_{1}}})^{x})
\]
for $x \geq 1$, where
\[
  \tilde{z_{1}} = \frac{\lambda - \sqrt{\lambda^{2} - 4pq}}{2q} \hspace{2em} (= \frac{1}{z_{2}}),
\]

\[
  \tilde{z_{2}} = \frac{\lambda + \sqrt{\lambda^{2} - 4pq}}{2q} \hspace{2em} (= \frac{1}{z_{1}}),
\]
To satisfy equation (\ref{eqn:linsystem4}), we must then take
\[
  b = - \frac{\tilde{z_{2}}}{\tilde{z_{1}}}(\frac{\lambda - q \tilde{z_{2}}}{\lambda - q \tilde{z_{1}}}).
\]

% TODO: Check the equatoin for \lambda = 2 sqrt{pq}
Because $0 < \tilde{z_{1}} < \tilde{z_{2}} < \frac{\lambda}{q}$, it follows that $(\eta (x) : x \geq 1)$ is a strictly positive row eigenvector solution to equation (\ref{eqn:etaeigen}), for each $\lambda > 2 \sqrt{pq}$. On the other hand, for $\lambda = 2 \sqrt{pq}$, $\eta (x) = (\frac{p}{q})^{x/2} x$ for $x \geq 1$, so a positive row eigenvector also exists at this value of $\lambda$.

We conclude that in this example, there exist positive row and column eigenvector solutions to the eigensystem for  $B$ at all $\lambda \in [2\sqrt{pq}, \infty)$. So, the Perron-Frobenius eigenvalue and eigenvectors are not uniquely defined by positivity considerations in this infinite-dimensional irreducible setting.

We now explore this region of $\lambda$ in more detail. Suppose $u$ is a column eigenvector of $B$ satisfying equation (\ref{eqn:eigenrel}). Then, the matrix $\tilde{P} = (\tilde{P}(x, y): x, y \in S)$ with entries given by
\[
 \tilde{P}(x, y) = \frac{B(x,y)u(y)}{\lambda u(x)}
\]
for $x, y \in S$ is an irreducible stochastic matrix. Note that the Markov chain $X$ associated with $\tilde{P}$ is again of birth-death form, with $\tilde{P}(1,2) = 1$ and
\[
  \tilde{P} (x, x+1) = \frac{p u(x+1)}{\lambda u(x)},
\]

\[
  \tilde{P} (x, x-1) = \frac{q u(x-1)}{\lambda u(x)},
\]
for $x \geq 2$. For $\lambda > 2 \sqrt{pq}$, $u(x) \sim z_{2}^{x}$ as $x \to \infty$, where we use the notation of $a(x) \sim b(x)$ as $x \to \infty$ to denote the fact that $a(x)/b(x) \to 1$ as $x \to \infty$.

Consequently,
\[
  \tilde{P}(x, x+1) \to \frac{p}{\lambda} z_{2} = \frac{1}{2} (1 + \sqrt{1 - \frac{4pq}{\lambda^{2}}})
\]
and
\[
  \tilde{P}(x, x-1) \to \frac{q}{\lambda} z_{2}^{-1} = \frac{1}{2} (1 - \sqrt{1 - \frac{4pq}{\lambda^{2}}})
\]
as $x \to \infty$ Hence, $\tilde{P}$ is the one-step transition matrix of a transient Markov chain for $\lambda > 2 \sqrt{pq}$. It follows that
\[
  \sum_{n=1}^{\infty} \tilde{P}^{n}(x, y) < \infty
\]
for each $x, y \in S$; see, for example, p$.\!$ 24 \cite{Chung66}. This immediately implies that
\begin{align}
  \label{eqn:bnsum}
  \sum_{n}^{\infty} \lambda^{-n} B^{n}(x,y) < \infty
\end{align}
for each $x, y \in S$. On the other hand, for $\lambda = 2 \sqrt{pq}$
\[
  \tilde{P}(x, x+1) = \frac{1}{2} (\frac{x+1}{x})
\]
and
\[
  \tilde{P}(x, x-1) = \frac{1}{2} (\frac{x-1}{x})
\]
for $x \geq 2$, while $\tilde{P}(1, 2) = 1$. Let $X = (X_{n}: n \geq 0$) be the Markov chain having this transition matrix, and let $T_{y} = \inf \{ n \geq 0: X_{n} = y \}$ be the hitting time of state $y$. If $v(x) = P_{x}(T_{1} < \infty)$, then $v = (v(x):  x \geq 1)$ satisfies the linear system
\[
  v(x) = \tilde{P}(x, x+1) v(x+1) + \tilde{P}(x, x-1) v(x-1)
\]
for $x \geq 2$, subject to $v(1) = 1$. The solution is
\[
  v(x) = \frac{\sum_{x}^{\infty}\frac{1}{y(y+1)}}{{\sum_{1}^{\infty}\frac{1}{y(y+1)}}}
\]

for $x \geq 1$; see p$.\!$ 31 \cite{HPS72}.  Since $\sum_{1}^{\infty}y^{-2} < \infty$, evidently $v(2) < 1$, so that $\tilde{P}$ is the transition matrix of a transient Markov chain at $\lambda = 2 \sqrt{pq}$. Consequently, equation (\ref{eqn:bnsum}) is also valid at $\lambda = 2 \sqrt{pq}$.

To explore the behavior of the sum appearing in equation (\ref{eqn:bnsum}) for $\lambda < 2 \sqrt{pq}$, we note that
\begin{align*}
  B^{2n + 2}(x, x) & = P_{x}(X_{2n+2} = x, T > 2n+2) \\
                   & \geq P(x, x+1) P_{x+1}(T_{x} = 2n+1) \\
                    & = p P_{1}(T_{0} = 2n+1), 
\end{align*}
where $X$ here is the Markov chain associated with the slotted time queuing model. This latter probability is a well known first passage time quantity and is given by
\begin{align*}
  P_{1} (T_{0} = 2n+1) = \frac{1}{2n+1} {{2n+1} \choose {n+1}} p^{n}q^{n+1}; 
\end{align*}
see, for example, p$.\!$  352 of \cite{Feller50}. Sterling's approximation (see p$.\!$  52 of \cite{Feller50}) implies that
\[
  P_{1} (T_{0} = 2n+1) \sim c n^{-\frac{3}{2}}(2\sqrt{pq})^{2n+1}
\]
as $n \to \infty$, for some constant $c \in (0, \infty)$. Consequently,
\begin{align}
  \label{eqn:bninfty}
  \sum_{n=0}^{\infty}\lambda^{-n}B^{n}(x,x) = \infty
\end{align}
for  $\lambda < 2 \sqrt{pq}$.

This discussion makes clear that while positive eigenvalues and eigenvectors exist for all $\lambda \geq 2 \sqrt{pq}$, $\lambda = 2 \sqrt{pq}$ holds special significance; compare equation (\ref{eqn:bnsum}) to (\ref{eqn:bninfty}). This state of affairs holds for general infinite-dimensional irreducible sub-stochastic matrices $G = (G(x,y): x, y \in S)$.
In particular, for all such matrices, there exists a finite-valued non-negative $R$ such that for all $x, y \in S$,
\[
  \sum_{n=0}^{\infty} \beta^{n} G^{n}(x, y) < \infty
\]
for $\beta < R$, whereas
\[
  \sum_{n=0}^{\infty} \beta^{n} G^{n}(x, y) = \infty
\]
for $\beta > R$ (see p$.\!$  200-201 of \cite{Seneta81}); the quantity $R$ is called the {\em convergence parameter} of $G$, and it is therefore natural to view $R^{-1}$ as the Perron-Frobenius eigenvalue $\lambda$ for such a matrix $G$. However, at this level of generality, there is no universal guarantee that corresponding positive column eigenvectors for $G$ will exist, either at $\lambda = R^{-1}$ or $\lambda > R^{-1}$. The example we have studied in this section is intended to illustrate the fact that while no general guarantee exists, such eigenvectors can potentially exist at all $\lambda \geq R^{-1}$ in this setting.

Under certain additional assumptions on $G$, there exist positive vectors $c = (c(x) : x \in S)$ and $d = (d(x): x \in S)$ for which
\begin{align}
  \label{eqn:gnsum}
  G^{n}(x, y) \sim \lambda^{n} c(x)d(y)
\end{align}
as $n \to \infty$, analogous to the behavior manifested in the setting of finite-dimensional irreducible aperiodic non-negative matrices. In the next section, we provide a simple probabilistic proof for equation (\ref{eqn:gnsum}) under probabilistically appropriate sufficient conditions.

\section{A Probabilistic Proof of the Perron-Frobenius Theorem}
\label{sec:3pftheorem}

Let $G = (G(x, y):  x, y \in S$ be an irreducible non-negative matrix in which $S$ can be either finite or countably infinite. We assume that
\[
  s = \sup_{x \in S} \sum_{y}G(x,y) < \infty,
\]
so that by passing to $B = (B(x,y): x, y \in S)$ with
\[
  B(x,y) = G(x,y)/s,
\]
we may assume, without loss of generality, that we are dealing with an irreducible sub-stochastic matrix.

Our probabilistic representation depends on augmenting the state space $S$. Specifically, for $\Delta \notin S$, let $S_{\Delta} = S \cup \{ \Delta\}$, and put
\[
  P(x, y) =
  \left\{
    \begin{array}{ll}
      B(x,y) & , x, y \in S \\
      1 - \sum_{y \in S}B(x,y) &, x \in S, y = \Delta \\
      1 &, x, y = \Delta \\
      0 &, x = \Delta, y \in S.
    \end{array}
    \right.
\]
Then, $P = (P(x, y): x, y \in S_{\Delta})$ is a stochastic matrix on $S_{\Delta}$. To construct a positive eigenvalue $\lambda_{*}$ and corresponding positive eigenvectors for $B$, we let $X = (X_{n} : n \geq 0)$ be the Markov chain on $S_{\Delta}$ having transition matrix $P$, and we let $T = \inf \{ n \geq 0; X_{n} = \Delta \}$ be the hitting time of $\Delta$. Fix $z \in S$ and let $\tau = \inf\{ n\geq 1: X_{n} = z  \}$ be the first return time to $z$. Note that the strong Markov property (SMP) implies that $X$ regenerates at visits to $z$, in the sense that $(X_{\tau + n} : n \geq 0)$ has the same distribution as $X$ under $P_{z}$, and is independent of $(\tau, X_{j} : 0 \leq j < \tau)$.

For $x \in S_{\Delta}$, let $E_{x}(\cdot) \triangleq E (\cdot | X_{0} = x)$. For $x \in S$, let $T_{x} = \inf \{ n \geq 0: X_{n} = x\}$  and $\tau_{x} = \inf\{ n \geq 1: X_{n} = x\}$. We are now ready to state our key assumption, under which a Perron-Frobenius eigenvalue and column/row eigenvectors will exist.
\\

\newtheorem*{ass1}{A1}\label{mass1}

\begin{ass1}

There exists $\theta \geq 0$ such that
\[
  E_{z}e^{\theta \tau}I(T > \tau) = 1.
\]
\end{ass1}

In {A1}, we use the interpretation
\[
  I(T > \tau) = \sum_{k = 1}^{\infty}I(\tau = k, T > k)
\]
so that, in particular, $\tau$ is finite-valued on $\{ T > \tau \}$ (even when $T = \infty)$.

Set
\begin{align*}
  \lambda_{*} & = e^{- \theta}, \\
  u_{*}(x) & = E_{x}e^{\theta \tau}I(T > \tau),\\
  \eta_{*}(y) & = E_{z}\sum_{j=0}^{\tau-1}e^{\theta j}I(X_{j} = y, T > j),
\end{align*}
for $x, y \in S$.
\\

\newtheorem{theorem}{Theorem}

\begin{theorem}
  Under  {A1}, $u_{*} = (u_{*}(x) : x \in S)$ and $\eta_{*} = (\eta_{*}(y) : y \in S) $ are positive finite-valued column and row vectors respectively, and
  \begin{align}
    \label{eqn:th1}
    B u_{*} & = \lambda_{*} u_{*}, \\
    \eta_{*} B & = \lambda_{*} \eta_{*}.
  \end{align}
\end{theorem}

\begin{proof}
Note that by conditioning on $X_{1}$, we observe that for $x \in S$,
\begin{align}
%TODO: B1  
\label{eqn:1step}
u_{*}(x) = e^{\theta}\sum_{y \in S, y \neq z}B(x, y)u_{*}(y) + e^{\theta} B(x,z). 
\end{align}
In view of A1, $u_{*}(z) = 1$, so that equation (\ref{eqn:1step}) implies that $u_{*} = e^{\theta}Bu_{*}$. As for $\eta_{*}$, suppose $y \neq z$. Since $X_{0} \neq y$ under $P_{x}$,
\begin{align*}
  \eta_{*}(y) & = E_{z} \sum_{j=0}^{\tau -1} e^{\theta j} I(X_{j} = y, T > j) \\
              & = E_{z} \sum_{j=1}^{\tau - 1} e^{\theta j}I(X_{j} = y, T > j) \\
              & = e^{\theta}E_{z}\sum_{k=0}^{\infty} e^{\theta k}I(X_{k+1} = y, T \wedge \tau > k+1)\\
              &  = e^{\theta} \sum_{x \in S}E_{z}\sum_{k=0}^{\infty}e^{\theta k}I(X_{k} = x, T \wedge \tau > k)B(x,y)\\
              & = e^{\theta} \sum_{x \in S}E_{z}\sum_{k=0}^{\tau - 1}e^{\theta k}I(X_{k}= x, T >k) B(x,y) \\
              & = e^{\theta} \sum_{x \in S} \eta_{*}(x) B(x,y) \\
              & = e^{\theta}(\eta_{*}B)(y), 
\end{align*}
where we use the notational convention $a \wedge b \triangleq \min(a, b)$. On the other hand,
\begin{align*}
  e^{\theta} \sum_{x \in S} \eta_{*}(x) B(x, z) & = e^{\theta} \sum_{x \in S} E_{z} \sum_{j=0}^{\infty}e^{\theta j}I(X_{j} = x,T \wedge \tau > j) B(x, z) \\
     & = E_{z}\sum_{k=1}^{\infty} e^{\theta k} I(X_{k} = z, T > k, \tau > k-1) \\
     & = E_{z} \sum_{k=1}^{\infty}e^{\theta k}I(\tau = k, T >k) \\
     & = E_{z}e^{\theta \tau}I(T > \tau) = 1 \\
     & = E_{z} \sum_{j=0}^{\tau -1} e^{\theta j}I(X_{j} = z, T > j) \\
     & = \eta_{*}(z), 
\end{align*}
where we used  {A1} for the third last equality, and the fact that under $P_{z}$, $X_{0} = z$ with $X_{j} \neq z$ for $1 \leq j < \tau$ for the second last equality. So, $\eta_{*} = e^{\theta}\eta_{*}B$.

\vspace{1em}
We now turn to the finiteness and positivity of $u_{*}$ and $\eta_{*}$. Set $\tau_{y} = \inf \{ n \geq 1: X_{n} = y \}$. Note that the irreducibility of $B$ implies that $P_{y}(\tau < \tau_{y} \wedge T) > 0$ and $P_{z}(\tau_{y} < \tau \wedge T) > 0$ for $y \in S$. So,
\begin{align*}
  u_{*}(y) & = E_{y}e^{\theta \tau}I(T > \tau) \\
           & \geq E_{y}e^{\theta \tau} I(\tau_{y} \wedge T > \tau) \\
           & \geq P_{y}(\tau_{y} \wedge T > \tau) > 0.
\end{align*}
Also, {A1} and the SMP imply that
\begin{align*}
  1 & = E_{z}e^{\theta \tau}I(T > \tau) \\
    & \geq E_{z}e^{\theta \tau_{y}}I(\tau_{y} < \tau \wedge T)E_{y}e^{\theta \tau}I(T > \tau) \\
    & \geq P_{z}(\tau_{y} < \tau \wedge T) u_{*}(y),
\end{align*}
so $u_{*}(y) < \infty$ for $y \in S$.
\\

As for $\eta_{*}$, note that for $y \neq z$,
\[
  \sum_{j=0}^{\tau -1}e^{\theta j} I(X_{j} = y, T > j) > 0
\]
on $\{ \tau_{y} < \tau \wedge T\}$. So, $\eta_{*}(y) > 0$ since $P_{z}(\tau_{y} < \tau \wedge T) > 0$. Since we showed earlier that $\eta_{*}(z) = 1$, $\eta_{*}$ is therefore positive. Also, the SMP implies that
\begin{align*}
  \infty > u_{*}(y)  = & E_{y} e^{\theta \tau} I(\tau < \tau_{y} \wedge T) \\
                       & + E_{y}e^{\theta \tau_{y}} I(\tau_{y} < \tau \wedge T) u_{*}(y),
\end{align*}
and hence it follows that $E_{y}e^{\theta \tau_{y}} I(\tau_{y} < \tau \wedge T)  <1$. So, another application of the SMP yields
\begin{align}
  \label{eqn:etastar}
  \eta_{*}(y) = E_{z}e^{\theta \tau_{y}} I(\tau_{y} < \tau \wedge T)E_{y}\sum_{j=0}^{\tau -1}e^{\theta j} I(X_{j} = y, T > j).
\end{align}
Note that $E_{z}e^{\theta \tau_{y}} I(\tau_{y} < T \wedge \tau) \leq E_{z}e^{\theta \tau} I(T > \tau) = 1$.
If $\beta(y) \triangleq E_{y} \sum_{j=0}^{\tau -1}e^{\theta j} I(X_{j} = y, T > j)$, we see that
\[
  \beta (y) = 1 + E_{y}e^{\theta \tau_{y}} I(\tau_{y} < \tau \wedge T) \beta(y).
\]
Since $E_{z}e^{\theta \tau_{y}} I(\tau_{y} < T \wedge \tau) < 1$, $\beta(y) < \infty$, and equation (\ref{eqn:etastar}) therefore implies that $\eta_{*}(y) < \infty$, proving that $\eta_{*}$ is finite valued.
\end{proof}

With an additional assumption, much more can be said about $\lambda_{*}$, $\eta_{*}$, and $u_{*}$. In particular, we can show that $\lambda_{*}^{-1}$ is indeed the convergence parameter of $B$, and $u_{*}$ and $\eta_{*}$ do indeed correspond to the Perron-Frobenius eigenvectors of the finite-dimensional theory.
\\

\newtheorem*{ass2}{A2}\label{mass2}
\begin{ass2}
  With $\theta$ defined as in A1, assume that
  \[
    E_{z}\tau e^{\theta \tau} I(T > \tau) < \infty.
  \]
\end{ass2}

Recall that the period $p$ of a non-negative irreducible matrix is always uniquely defined (as a {\em class property}) and finite-valued (even in the infinite-dimensional setting); see p$.\!$ 20 \cite{Nummelin84}.
\\

\begin{theorem}
Assume A1 and A2. Then:
\begin{enumerate}[label = (\roman*)]
\item If $u$ is a positive column vector satisfying $Bu = \lambda_{*}u$, then $u$ is a positive multiple of $u_{*}$.
\item If $\eta$ is a positive row vector satisfying $\eta B = \lambda_{*} \eta$, then $\eta$ is a positive multiple of $\eta_{*}$.
\item $\sum_{w \in S} \eta_{*}(w) u_{*}(w) < \infty$.
\item If $p$ is the period of $B$, then
  \[
    \frac{1}{p} \sum_{j=0}^{p-1} \lambda_{*}^{-pn - j}B^{pn + j}(x, y) \to \frac{u_{*}(x)\eta_{*}(y)}{\sum_{w \in S} u_{*}(w)\eta_{*}(w)}
  \]
as $n \to \infty$.  
\end{enumerate}
\end{theorem}

\begin{proof}
  We start by observing that
  \begin{align*}
    \sum_{x \in S}\eta_{*}(x) u_{*}(x) & = \sum_{x \in S}E_{z}\sum_{j=0}^{\tau - 1} e^{\theta j} I(X_{j} = x, T > j) u_{*}(x)\\
      & = \sum_{x\in S}E_{z}\sum_{j = 0}^{\infty} e^{\theta j}I(T \wedge \tau > j) I( X_{j} = x)E_{x}e^{\theta \tau}I(T > \tau) \\
      & = \sum_{x \in S}E_{z} \sum_{j = 0}^{\infty} e^{\theta \tau} I(T \wedge \tau > j, X_{j} = x) \\  
      & = E_{z} \sum_{j=0}^{\tau -1} e^{\theta \tau } I(T \wedge \tau > j) \\
      & = E_{z}\tau e^{\theta \tau} I(T > \tau) < \infty
  \end{align*}
  by A2.

  Set
  \[
    P_{*}(x, y) = \frac{B(x, y) u_{*}(y)}{\lambda_{*} u_{*}(x)}
  \]
  for $x, y \in S$. Observe that $P_{*}$ is an irreducible stochastic matrix that inherits the same periodicity as does $B$. If $\pi_{*}(x) = u_{*}(x)\eta_{*}(x)/\sum_{w \in S} u_{*}(w)\eta_{*}(w)$, then $\pi_{*} = (\pi_{*}(x): x \in S)$ is a probability on $S$, and
  \begin{align*}
    \sum_{x \in S}\pi_{*}(x) P_{*}(x, y) & = \sum_{x \in S} \frac{\eta_{*}(x) B(x, y)} {\lambda_{*}} u_{*}(y) \cdot \frac{1}{\sum_{w \in S} u_{*}(w) \eta_{*}(w)} \\
    & = \frac{\eta_{*}(y) u_{*}(y)} {\sum_{w \in S} u_{*}(w) \eta_{*}(w)} \\
    & = \pi_{*}(y), 
  \end{align*}
  so that $\pi_{*} = (\pi_{*}(x): x \in S)$ is a stationary distribution of $P_{*}$. Consequently, the Markov chain $X = (X_{n}: n \geq 0)$ having transition matrix $P_{*}$ is positive recurrent.

  Suppose that $u = (u(x) : x \in S)$ is a positive solution of $Bu = \lambda_{*}u$. Then, if $w(x) \triangleq u(x)/u_{*}(x)$ for $x \in S$,
  \begin{align}
    \label{eqn:eigw}
    \sum_{y \in S} P_{*}(x, y) w(y) & = w (x)
  \end{align}
  for $x \in S$. It follows from equation (\ref{eqn:eigw}) that $(w(X_{n}): n \geq 0)$ is a positive martingale adapted to the history of $X$. The supermartingale convergence theorem (see, for example, p$.\!$ 335 \cite{Chung74}) then establishes that there exists a finite-valued random variable (r.v.) $M_{\infty}$ such that
  \[
    w(X_{n}) \to M_{\infty} \hspace{2em} P_{x} \hspace{1em} a.s.
  \]
  as $n \to \infty$. Since $X$ visits each state $x \in S$ infinitely often (due to recurrence), the a.s. (almost sure) convergence implies that $w(\cdot)$ must be constant. In other words, $u$ must be a positive multiple of $u_{*}$.

  Suppose that $\eta = (\eta(y):  y \in S)$ is a positive solution of $\eta B = \lambda_{*} \eta$. Put $\pi(x) = \eta(x) u_{*}(x)$, and note that $\pi = (\pi(x) : x \in S)$ satisfies
  \[
    \sum_{x \in S}\pi(x) P_{*}(x, y) = \pi(y)
  \]
  for $y \in S$, so that $\pi$ is a positive solution of the equilibrium equation for $P_{*}$. This implies that $\pi$ must be a multiple of $\pi_{*}$; see Theorem 9.2 \cite{Nummelin84}. This, in turn, proves that $\eta$ must be a positive multiple of $\eta_{*}$.

  Finally, the positive recurrence of $X$ implies that for each $x, y \in S$,
  \[
    \frac{1}{p} \sum_{j=0}^{p-1}P_{*}^{np + j}(x, y) \to \pi_{*}(y)
  \]
  as $n \to \infty$; (see p$.\!$  98 of \cite{Gantmacher59} for the proof holds when $|S| < \infty$; an analogous proof holds when $|S| = \infty$). But this is exactly equivalent to part {\em iv)} of the theorem.
\end{proof}

An easy consequence of part {$iv$)} is that $\lambda_{*}$ is the convergence parameter of $B$. Thus, $\lambda_{*}$ describes the decay rate of  $B^{n}$, and $u_{*}$ and $\eta_{*}$ arise naturally in describing the large $n$ behavior of $B^{n}$; see part {\em iv} of the theorem . Thus, $\lambda_{*}$, $u_{*}$, and $\eta_{*}$ are the natural Perron-Frobenius eigenvalue, column eigenvector, and row eigenvector associated with $B$.

We now turn to a simple sufficient condition guaranteeing {A1} and {A2}. Let
\begin{align*}
  \theta_{1} & = \sup \{ \gamma :  E_{z}e^{\gamma T} < \infty \}, \\
   \theta_{2} & = \sup \{ \gamma : E_{z} e^{\gamma (T \wedge \tau)} < \infty  \}.
\end{align*}

\begin{theorem}
If $\theta_{2} > \theta_{1}$, then A1 and A2 hold.
\end{theorem}

\begin{proof}
Note that for $\gamma \in \mathbf{R}$, the SMP implies that
\begin{align}
\label{eqn:gammaSMP}
  E_{z}e^{\gamma T} = E_{z}e^{\gamma \tau}I(T > \tau) E_{z}e^{\gamma T} + E_{z}e^{\gamma T}I(T < \tau).
\end{align}
Choose $\theta_{0} \in (\theta_{1}, \theta_{2})$. By definition of $\theta_{1}$ and $\theta_{2}$, it follows that $E_{z}e^{\theta_{0} T} = \infty$ but $E_{z}e^{\theta_{0} (T \wedge \tau)} < \infty$. In view of equation (\ref{eqn:gammaSMP}), we are then led to the conclusion that $E_{z}e^{\theta_{0}\tau}I(T > \tau) \geq 1$. Since $h(\gamma) \triangleq E_{z}e^{\gamma \tau} I (T > \tau) \leq E_{z}e^{\gamma (T \wedge \tau)} < \infty$, the function $h(\cdot)$ is finite-valued, strictly increasing, and infinitely differentiable on $(- \infty, \theta_{2}]$. Hence, there exists a unique $\theta < \theta_{0}$ such that $E_{z}e^{\theta \tau} I(T > \tau) = 1$, proving {A1}.

Furthermore, by an argument analogous to that used in studying moment generating functions (see, for example, p$.\!$  119 of \cite{Billingsley79}),
\[
  h'(\gamma) = E_{z}\tau e^{\gamma \tau} I(T > \tau),
\]
for $\gamma < \theta_{2}$, so that $E_{z}\tau e^{\theta \tau} I(T > \tau) = h'(\theta) < \infty$, proving {A2}. 
\end{proof}

Given the importance of {A1} and {A2}, it is of some interest to determine whether {A1} and {A2} are {\em solidarity properties}, in the sense that if they hold at one $z \in S$, then they hold together at all $z \in S$. Let $P_{*}(\cdot)$ be the probability in the path-space of $X$ under which $X$ has transition matrix $(P_{*}(x, y) : x, y \in S)$, and let $E_{*}(\cdot)$ be the corresponding expectation operator. Then, for any function $f : S^{n+1} \to \mathbf{R}_{+}$,
\begin{eqnarray*}
  E_{*}[f(X_{0}, \hdots, X_{n}) | X_{0} = x] & \hspace{9em} & \hspace{9em}
\end{eqnarray*}
\begin{align*}
  \hspace{2em } &   =  \sum_{x_{1},\hdots, x_{n}}f(x, x_{1}, \hdots, x_{n}) P_{*}(x, x_{1})  P_{*}(x_{1}, x_{2})\cdots  P_{*}(x_{n-1}, x_{n}) \\
    &   =  \sum_{x_{1},\hdots, x_{n}}f(x, x_{1}, \hdots, x_{n}) B(x, x_{1})  B(x_{1}, x_{2}) \cdots  B(x_{n-1}, x_{n}) \frac{u_{*}(x_{n})}{u_{*}(x)} \lambda_{*}^{-n} \\
   &  =   E_{x} e^{\theta n}f(X_{0}, \hdots, X_{n}) I(T > n) \frac{u_{*}(X_{n})}{u_{*}(X_{0})}.
 \end{align*}
Hence, since $X_{n} = y$ on $\{\tau_{y} = n \}$,
\begin{align}
  \label{eqn:ptaun}
  P_{*}(\tau_{y} = n | X_{0} = y) = E_{y}e^{\theta n} I(\tau_{y} = n, T > n).
\end{align}
Since $X$ is recurrent under $P_{*}$,
\[
  1 = P_{*}(\tau_{y} < \infty | X_{0} = y) = E_{y}e^{\theta \tau_{y}}I(\tau_{y} < T),
\]
proving that {A1} holds with $y \in S$ replacing $z$.
As for {A2}, equation (\ref{eqn:ptaun}) yields
\begin{align*}
  E_{y}\tau_{y}e^{\theta \tau_{y}}I(T > \tau_{y}) & = \sum_{n=1}^{\infty}n E_{y}e^{\theta n}I( \tau_{y} = n, T > n)\\
  & = \sum_{n=1}^{\infty} n P_{*}(\tau_{y} = n | X_{0} = y) \\
  & = E_{*}[\tau_{y} | X_{0} = y].
\end{align*}

But since $X$ is positive recurrent under $P_{*}$ (and positive recurrence is a solidarity property),
\[
  E_{*}[\tau_{y} | X_{0} = y] < \infty
\]
for all $y \in S$, so that {A2} also holds for all $y \in S$. We have therefore established the following property.
\\

\newtheorem*{prop1}{Proposition 1}\label{mprop1}
\begin{prop1}
  If {A1} and {A2} hold at one $z \in S$, then {A1} and {A2} hold at each $y \in S$.
\end{prop1}

So, {A1} and {A2} are indeed {solidarity properties}, so we are free to choose any $z \in S$ to verify these conditions.

Our final goal in this section is to provide an easily verifiable condition on $B$ guaranteeing that $\theta_{2} > \theta_{1}$ (as in the statement of Theorem {3}. Specifically, consider the hypothesis:
\\

\newtheorem*{ass3}{A3}\label{mass3}
\begin{ass3}
  There exists $v \in S$ and $c_{1}$, $c_{2} \in \mathbf{R}_{+}$ such that
  \[
    0 < c_{1} = \inf_{x \in S} \frac{B(x, y)}{B(v, y)} \leq \sup_{x \in S} \frac{B(x, y)}{B(v,y)} = c_{2} < \infty.
  \]
\end{ass3}
Assumption {A3} requires that the rows of $B$ be similar, in the sense that they must all be within a uniform constant multiple of one another. This assumption is automatic when $B$ is a finite-dimensional strictly positive matrix, but is a (very) strong condition when $B$ is infinite-dimensional.

Put $\psi(y) = B(v, y) / \sum_{w}B(v, w)$ for $y \in S$ and set $\psi = (\psi(y) : y \in S)$. To exploit {A3}, we need to use a regenerative structure for $X$ based on the distribution of $\psi$, rather than regenerations based on the consecutive hitting times of $z$. In particular, note that by definition of $c_{1}$, $B(x, y) \geq \delta \psi(y)$ for $x, y \in S$ with $\delta = c_{1} \sum_{w}B(v, w)$ so that
\begin{align}
  \label{eqn:psi}
  B(x, y) = \delta \psi(y) + \tilde{B}(x, y)
\end{align}
for $x, y \in S$, where $\tilde{B}(x, y) \triangleq B(x, y) - \delta \psi(y) \geq 0$ (so that $\tilde{B} = (\tilde{B}(x, y) : x, y \in S)$ is a sub-stochastic matrix). By summing both sides of equation (\ref{eqn:psi}) over $y \in S$, we conclude that $\delta \leq 1$. Hence, we may view a transition of $X$ at time $n$ as follows:\\
\begin{itemize}
\item With probability $\delta$, distribute $X_{n+1}$ according to $\psi$; \\
\item With probability $\tilde{B}(X_{n}, y)$, set $X_{n+1} = y$ ; \\
\item With probability $1 - \sum_{y \in S} B(X_{n}, y)$, set $X_{n+1} = \Delta$.\\
\end{itemize}

At times $\tau$ at which $X$ distributes itself according to $\psi$, $X$ regenerates and has distribution $\psi$, independent of $X_{0},\hdots,X_{\tau - 1}$, and $\tau$. Let $E_{\psi}(\cdot) \triangleq \sum_{x} \psi(x) E_{x}(\cdot)$ and $ P_{\psi}(\cdot) \triangleq \sum_{x} \psi(x) P_{x}(\cdot)$.
Put
\begin{align*}
  \theta_{1} & = \sup \{ \theta : E_{\psi} e^{\theta T} < \infty \}, \\
  \theta_{2} & = \sup \{ \theta : E_{\psi} e^{\theta (T \wedge \tau)} < \infty \}.
\end{align*}
Theorems {1} through {3} generalize naturally to the randomized regenerations $\tau$ just introduced.
\\
\begin{theorem}
  If $\theta_{2} > \theta_{1}$, then there exists a solution $\theta$ to the equation
  \begin{align}
    \label{eqn:psitau}
    E_{\psi} e^{\theta \tau} I(T > \tau) = 1.
  \end{align}

Put $\lambda_{*} = e^{- \theta}$, and set
\[
  u_{*}(x) = E_{x}e^{\theta \tau} I(T > \tau),
\]
\[
  \eta_{*}(y) = E_{\psi} \sum_{j=1}^{\tau -1} e^{\theta j}I(X_{j} = y, T > \tau).
\]

Then:

\begin{enumerate}[label = (\roman*)]
\item $u_{*} = (u_{*}(x) : x \in S)$ is a positive finite-valued solution of $Bu = \lambda_{*} u$. Any other positive finite-valued solution must be a positive multiple of $u_{*}$.
\item $\eta_{*} = (\eta_{*}(y): y \in S)$ is a positive finite-valued solution of $\eta B = \lambda_{*} \eta$. Any other positive finite-valued solution must be a positive multiple of $\eta_{*}$.
\item $\sum_{w \in S}\eta_{*}(w) u_{*}(w) < \infty$.
\item If $p$ is the period of $B$, then
  \[
    \frac{1}{p}\sum_{j=0}^{p-1} \lambda_{*}^{-pn - j} B^{pn + j}(x, y) \to \frac{u_{*}(x) \eta_{*}(y)}{\sum_{w \in S} \eta_{*}(w) u_{*}(w)}
  \]
  as $n \to \infty$.
\end{enumerate}
\end{theorem}

\begin{proof}
  The proofs of Theorems {2} and {3} extend to this setting without change. It follows that there exists a root $\theta$ to equation (\ref{eqn:psitau}). Note that equation (\ref{eqn:psitau}) implies that
  \begin{align*}
    u_{*}(x) & = e^{\theta} P_{x}(\tau = 1) + e^{\theta} \sum_{y \in S} \tilde{B}(x, y) u_{*}(y) \\
    & = e^{\theta} \delta + e^{\theta} \sum_{y \in S}\tilde{B}(x, y) u_{*}(y) \\
    & = e^{\theta} \delta \sum_{y \in S} \psi(y) u_{*}(y) + e^{\theta} \sum_{y \in S} \tilde{B} (x, y) u_{*}(y) \\
    & = e^{\theta} \sum_{y \in S}[\delta \psi(y) + \tilde{B}(x, y)] u_{*}(y) \\
    & = e^{\theta} \sum_{y \in S} B(x, y) u_{*}(y) = \lambda_{*}^{-1} (Bu_{*})(x)
  \end{align*}
  for $x \in S$, verifying the fact that $u_{*}$ is a column eigenvector satisfying $Bu = \lambda_{*} u$. As for $\eta_{*}$,
  \begin{align*}
    e^{\theta} \sum_{x \in S} \eta_{*}(x) B(x, y) & = E_{\psi} \sum_{x \in S} \sum_{j=0} ^{\infty} I(X_{j} = x, T \wedge \tau > j) e^{\theta (j+1)}B(x,y) \\
    & = E_{\psi} \sum_{x \in S} \sum_{j=1} ^{\infty} e^{\theta j} I(X_{j-1} = x, T \wedge \tau > j - 1) (\tilde{B}(x, y) + \delta \psi(y)) \\
    & = E_{\psi}\sum_{j=1}^{\tau -1}e^{\theta j} I(X_{j} = y, T > j) +  E_{\psi} e^{\theta \tau} I(T > \tau) \psi(y) \\ 
    & = \psi(y) + E_{\psi} \sum_{j=1}^{\tau - 1}e^{\theta j} I(X_{j} = y, T > j) \\
    & = E_{\psi} \sum_{j=0}^{\infty} e^{\theta j}I(X_{j} = y, T > j) \\
    & = \eta_{*} (y),
  \end{align*}
proving that $\eta_{*}$ is a solution of $\eta B = \lambda_{*} \eta$.

As for the positivity of $u_{*}$, this follows from the fact that $u_{*}(y) \geq e^{\theta} P_{y}(\tau = 1) = \delta e^{\theta} > 0$. For $\eta_{*}$, note that $\eta_{*}(y) \geq E_{\psi} I(T_{y} < \tau \wedge T)e^{\theta T_{y}} > 0$ if $P_{\psi}(T_{y} < T \wedge \tau) > 0$. Suppose, however, that $P_{\psi}(T_{y} < T \wedge \tau) = 0$. Then, the regenerative structure of $\tau$ guarantees that $P_{\psi}(T_{y} < T \wedge \tau_{n}) = 0$, where $\tau_{n}$ is the time of the $n^{th}$ randomized regeneration. Sending $n \to \infty$, we conclude that $P_{\psi} (T_{y} < T) = 0$, contradicting the irreducibility of $B$.

  To obtain the finiteness of $u_{*}$, observe that
  \begin{align*}
    1 & = E_{\psi}e^{\theta \tau}I(T > \tau) \\
    & \geq P_{\psi}(T_{y} < T \wedge \tau)u_{*}(y),
  \end{align*}
  and use the just established positivity of $P_{\psi}(T_{y} <  T \wedge \tau)$. As for the finiteness of $\eta_{*}$, note that Theorems {2} and {3} establish that $E_{\psi} \tau e^{\theta \tau} I(T > \tau) < \infty$. But $\tau e^{\theta \tau} I(T > \tau) \geq \sum_{j=0}^{\tau -1 } e^{\theta j} I(X_{1} = y, T > j)$, yielding the result. The uniqueness of $u_{*}$ and $\eta_{*}$ follow as in Theorem 2. 
\end{proof}

We are now ready to prove our next result.
\\
\begin{theorem}
  If {A3} holds, then the conclusions of {\em Theorem 4} are in force.
\end{theorem}

\begin{proof}
  It remains only to show that $\theta_{2} > \theta_{1}$. Note that
  \begin{align}
    \label{eqn:pTtau}
    P_{x}(T \wedge \tau > n) = \sum_{x_{1},\hdots,x_{n}}\tilde{B}(x, x_{1})\tilde{B}(x_{1}, x_{2}) \cdots \tilde{B}(x_{n-1},x_{n}).
  \end{align}
  But {A3} implies that $B(x,y) \leq (c_{2}/c_{1}) \delta \psi(y)$, so $- \delta \psi(y) \leq (c_{1}/c_{2})B(x,y)$. Consequently,
  \[
    \tilde{B}(x, y) \leq (1 - \frac{c_{1}}{c_{2}}) B(x, y)
  \]
  for $x, y \in S$. Substituting this inequality into equation (\ref{eqn:pTtau}) yields
  \begin{align*}
    P_{x} (T \wedge \tau > n) & \leq (1 - \frac{c_{1}}{c_{2}})^{n} \sum_{x_{1},\hdots,x_{n}}B(x, x_{1}) \cdots B(x_{n-1}, x_{n}) \\
    & = (1 - \frac{c_{1}}{c_{2}})^{n} P_{x}(T > n).
  \end{align*}
  Since $\sum_{k=0}^{j-1} e^{\theta k}(e^{\theta} -1) = e^{\theta j} - 1$, evidently
  \begin{align*}
    E_{x}e^{\theta (T \wedge \tau)} & = 1 + (e^{\theta} -1 ) E_{x}\sum_{k=0}^{(T \wedge \tau) - 1}e^{\theta k}\\
    & = 1 + (e^{\theta } - 1) \sum_{k=0}^{\infty} e^{\theta k} P_{x}(T \wedge \tau > k) \\
    & \leq 1 + (e^{\theta} - 1) \sum_{k=0}^{\infty} e^{\theta k}(1 - \frac{c_{1}}{c_{2}})^{k} P_{x}(T > k) \\
    & = 1 + (e^{\theta} - 1) E_{x}\sum_{k=0}^{T -1} (e^{\theta }(1 - \frac{c_{1}}{c_{2}}))^{k} \\
    & = 1 + \frac{(e^{\theta} - 1)}{(e^{\hat{\theta}} - 1) }(E_{\psi}e^{\hat{\theta} T} - 1),
  \end{align*}
  where $\hat{\theta} = \theta + \log (1 - \frac{c_{1}}{c_{2}})$. We conclude that $ \theta_{2} \geq \theta_{1} - \log(1 - (c_{1}/c_{2})) > \theta_{1}$, proving the result.
\end{proof}

In fact, under {A3}, the period of $B$ is obviously one (since $B(y,y) > 0$ for each $y \in S$ for which $\psi(y) > 0$), so we conclude that {A3} ensures that
\[
  B^{n}(x, y) \sim \lambda_{*}^{n} \frac{u_{*}(x) \eta_{*}(y)}{\sum_{w \in S}u_{*}(w) \eta_{*}(w)}
\]
as $n \to \infty$. In other words, {A3} guarantees the validity of equation (\ref{eqn:gnsum}).

Actually, {A3} can be further generalized.
\\

\newtheorem*{ass3p}{A3$^\prime$}\label{mass3p}

\begin{ass3p}
  There exists $m \geq 1, v \in S$, and $c_{1}, c_{2} \in \mathbf{R}_{+}$ such that:
  \[
    0 < c_{1} = \inf_{x \in S} \frac{B^{m}(x, y)}{B^{m}(v, y)} \leq \sup_{x \in S} \frac{B^{m}(x, y)}{B^{m}(v, y)}  = c_{2} <\infty.
  \]
\end{ass3p}
Note that when $B$ is finite-dimensional and has period one, the irreducibility that we are requiring throughout this section guarantees that $B^{m}$ is strictly positive for some $m \geq 1$, and {A3$^\prime$} is then immediate.

Assuming {A3$^\prime$}, we observe that if we choose $y \in S$ so that $B^{m}(v, y) > 0$, then $B^{2m}(y,y) \geq B^{m}(y, y) B^{m}(y, y) \geq c_{1}^{2} [B(v, y)]^{2}$ and
\begin{align*}
   B^{2m +1}(y, y)  & \geq B^{m}(y, y) \sum_{x \in S}B(y, x) B^{m}(x, y) \\
       & \geq c_{1}^{2}B^{m}(v, y) \sum_{x}B(y,x) B^{m}(v, y)  > 0,
\end{align*} 
so the greatest common divisor of $\{ k \geq 1: B^{k}(y, y) > 0 \}$ equals 1. Hence, $B$ has period one, which implies that $B^{m}$ is irreducible. Consequently, Theorem {5} can be applied to $B^{m}$, so that the conclusions of Theorem {4} apply also to $B^{m}$.

Note also that if $u_{*}, \eta_{*}$, and $\lambda_{*}$ are as in Theorem {4} (when applied to $B^{m})$,
\[
  B^{m} u_{*} = \lambda_{*} u_{*}
\]
and
\[
  \eta_{*} B^{m} = \lambda_{*} \eta_{*}.
\]
But this implies that
\begin{align}
  \label{eqn:star}
  B^{m}(Bu_{*}) = \lambda_{*} B u_{*}
\end{align}
and
\begin{align}
  \label{eqn:2star}
  (\eta_{*}B)B^{m} = \lambda_{*} \eta_{*} B,
\end{align}
so that $\eta_{*}B$ and $Bu_{*}$ are positive row and column eigenvectors associated with $B^{m}$. The uniqueness of $u_{*}$ and $\eta_{*}$ (up to positive multiples) then implies that $Bu_{*} = a_{1} u_{*}$ and $\eta_{*}B = a_{2} \eta_{*}$ for some $a_{1}, a_{2} > 0$. Consequently, $B^{m}u_{*} = a_{1}^{m} u_{*}$ and $\eta_{*} B^{m} = a_{2}^{m} \eta_{*}$. The positivity of $u_{*}$ and $\eta_{*}$ then ensure that $a_{1}^{m} = \lambda_{*} = a_{2}^{m}$, due to equations (\ref{eqn:star}) and (\ref{eqn:2star}). So, $\lambda_{*}$, $u_{*}$, and $\eta_{*}$ satisfy
\[
  B u_{*} = \lambda_{*}^{1/m} u_{*},
\]
and
\[
  \eta_{*} B = \lambda_{*}^{1/m} \eta_{*},
\]
establishing the following result.
\\

\begin{theorem}
  If {A3$^\prime$} holds, then we can define $\lambda_{*}$, $u_{*}$, and $\eta_{*}$ by applying Theorem 4 to $B^{m}$. Furthermore,
  \[
    B^{n}(x, y) \sim \lambda_{*}^{n/m} \frac{u_{*}(x) \eta_{*}(y)}{\sum_{w \in S} u_{*}(w) \eta_{*}(w)}
  \]
as $n \to \infty$, and $u_{*}$ and $\eta_{*}$ are the unique (up to a positive multiple) column and row eigenvectors of $B$ associated with eigenvalue $\lambda_{*}^{1/m}$.  
\end{theorem}

\section{Extension to Continuous State Space}
\label{sec:4}

In this section, we briefly indicate how the theory of Section \ref{sec:3pftheorem} easily extends to continuous state space. In particular, suppose that $S$ is a metric space equipped with the measurable subsets induced by its metric topology, and suppose that $B = (B(x, dy): x, y \in S )$ is a non-negative sub-stochastic kernel, so that for each $x \in S$, $B(x, \cdot)$ is a non-negative measure on $S$ for which $B(x, S) \leq 1$. We further assume that there exists a probability $\nu(\cdot)$ on $S$ and $m \geq 1$ so that for $x, y \in S$,
\[
  B^{m}(x, dy) = b_{m}(x, y) \nu(dy)
\]
for some positive density $(b_{m}(x, y) : x, y \in S)$. Finally, we assume that:
\\
\newtheorem*{ass4}{A4}
\label{mass4}
\begin{ass4}
  There exists $v \in S$ and $c_{1}, c_{2} \in \mathbf{R}_{+}$ such that
  \[
    0 < c_{1} = \inf_{x \in S} \frac{b_{m}(x, y)}{b_{m}(v, y)} \leq \sup_{x \in S} \frac{b_{m}(x, y)}{b_{m}(v, y)} = c_{2} < \infty.
  \]
\end{ass4}
For (measurable) $C \subseteq S$, set
\[
  \psi(C) = \frac{\int_{C}b_{m}(v, y) \nu(dy)}{\int_{S}b_{m}(v, y) \nu(dy)}.
\]
As in Section \ref{sec:3pftheorem}, we can now write
\begin{align}
  \label{eqn:4.1conv}
  B^{m}(x, dy) = \delta \psi(dy) + \tilde{B}(x, dy)
\end{align}
for $x, y \in S$, and let $\tau$ be the first $k \geq 1$ at which $X_{km}$ has the distribution $\psi$. Under A4, the proof of Theorem 5 holds without change, and we therefore obtain the existence of $\theta \geq 0$ such that
\[
  E_{\psi}e^{\theta \tau} I(T > \tau m) = 1
\]
with
\[
  E_{\psi} \tau e^{\theta \tau} I(T > \tau m) < \infty.
\]
Furthermore, if we set $\lambda_{*} = e^{- \theta}$ and
\[
  u_{*}(x) = E_{x}e^{\theta \tau} I(T > \tau m),
\]
\[
  \eta_{*}(C) = E_{\psi} \sum_{j = 0}^{\tau -1} e^{\theta j} I(X_{jm} \in C, T > jm)
\]
for $x \in S$ and (measurable) $C \subseteq S$, the following theorem holds.

\vspace{1em}
\begin{theorem}
  If A4 holds, then for $x \in S$ and (measurable) $C \subseteq S$ for which $\eta_{*}(C) > 0$,
  \[
    B^{n}(x, C) \sim \lambda_{*}^{n/m} \frac{u_{*}(x) \eta_{*}(C)}{\int_{S}u_{*}(w) \eta_{*}(dw)}
  \]
  as $n \to \infty$.
\end{theorem}

\begin{proof}
  Under A4, the proof of Theorem 4 shows that
  \begin{align}
    \label{eqn:BmEig}
    B^{m} u_{*} & = \lambda_{*} u_{*} 
  \end{align}
  and
  \begin{align}
    \label{eqn:BmEigEta}
    \eta_{*}B^{m} & = \lambda_{*} \eta_{*}.
  \end{align}
  The proof of Theorem 4 also shows that both $u_{*}(\cdot)$ and $\eta_{*}(\cdot)$ are finite-valued. We further note that $u_{*}(x) \geq e^{\theta} P_{x}(T > m, \tau = 1) = \lambda^{-1} \delta $, so $u_{*}(\cdot)$ is bounded below by a positive constant. Also, A4 and equation (\ref{eqn:BmEig}) imply that
  \begin{align*}
    \lambda_{*}u_{*}(x) & = (B^{m}u_{*})(x) \\
           & = \int_{S}b_{m}(x,y) u_{*}(y) \nu(dy) \\
           & \leq c_{2} \int_{S}b_{m}(v,y) u_{*}(y) \nu(dy) \\
           & = c_{2} (B^{m}u_{*})(v) \\
           & = c_{2} \lambda_{*} u_{*}(v)
  \end{align*}
so $u_{*}(x) \leq c_{2}u_{*}(v)$.
  
  Set
  \[
    P_{*}(x, dy) = \frac{B^{m}(x, dy) u_{*}(y)}{\lambda_{*} u_{*}(x)}
  \]
  for $x, y \in S$, and note that $P_{*} = (P_{*}(x, dy) : x, y \in S)$ is (as in Theorem 2) a stochastic kernel. Furthermore, equation (\ref{eqn:4.1conv}) and the upper bound on $u_{*}$ imply that
  \begin{align}
    \label{eqn:4.2lbd}
    P_{*}(x, dy) \geq \frac{\delta \psi(dy) u_{*}(y)}{\lambda_{*} u_{*}(x)} \geq \frac{\delta}{\lambda_{*}} \frac{\psi(dy)u_{*}(y)}{c_{2}u_{*}(v)}
  \end{align}
  uniformly in $x \in S$. It follows that $X$ is a uniformly ergodic Markov chain under $P_{*}$; see p$\!.$ 394 of \cite{MT2009} for the definition and basic properties. As in Section 3, $\pi_{*}$ is a stationary distribution of $P_{*}$, where
  \[
    \pi_{*}(dx) = \frac{u_{*}(x) \eta_{*}(dx)}{\int_{S} u_{*}(w)  \eta_{*}(dw)}
  \]
  for $x \in S$. The uniform ergodicity of $P_{*}$ implies that $\pi_{*}$ is the unique sigma-finite nonnegative solution of $\pi P_{*} = \pi$ (up to positive multiples; see p$\!.$ 234 of \cite{MT2009}) and this implies, as in Section 3, that $\eta_{*}$ is the unique (up to positive multiples) solution of $\eta B^{m} = \lambda_{*} \eta $. Furthermore, if $u$ is a nonnegative solution of $B^{m} u = \lambda_{*} u$, then $P_{*}w = w$, where $w(x) = u(x)/u_{*}(x)$ for $s \in S$.
Again, since $(w(X_{n}) : n \geq 0)$ is a non-negative supermartingale, $w(X_{n}) \to M_{\infty}$ a.s. as $n \to \infty$, where $M_{\infty}$ is finite-values a.s. Since $X$ hits every set with positive $\nu$-measure infinitely often a.s., this implies that there can not exist $d \in \mathbf{R}$ and $\epsilon > 0$ such that $\{x : w(x) < d - \epsilon \}$ and $\{ x : w(x) > d + \epsilon \}$ both have positive $\nu$-measure. Consequently, $w$ nust be a constant $\nu$-a.e. In view of A4, this implies that $P_{*}w$ must be constant, so that $w$ is constant. This proves that $u_{*}$ is the unique positive solution of $B^{m} = \lambda_{*} u$ (up to positive multiples). 
  Consequently, as in the argument leading to Theorem 6, we find that
  \begin{align*}
    B u_{*} = \lambda_{*}^{1/m} u_{*}
  \end{align*}
  and
  \begin{align*}
    \eta_{*} B = \lambda_{*}^{1/m} \eta_{*}.
  \end{align*}
  We can now define the stochastic kernel,
  \[
    \tilde{P}_{*}(x, dy) = \frac{B(x, dy) u_{*}(y)}{ \lambda_{*}^{1/m} u_{*}(x)}
  \]
  for $x, y \in S$ and note that $\tilde{P}_{*}^{m} = P_{*}$. Thus, equation (\ref{eqn:4.2lbd}) ensures that
  \[
    \tilde{P}_{*}^{m}(x, dy) \geq \psi(dy) u_{*}(y)
  \]
  uniformly in $x \in S$, so that if $X$ evolves on $S$ under $\tilde{P}_{*}$, $X$ is uniformly ergodic. Since $\pi_{*}$ is also stationary for $\tilde{P}_{*}$, this implies that
  \[
    \tilde{P}_{*}^{n}(x, C) \to \pi_{*}(c)
  \]
  as $n \to \infty$, which proves the theorem.
\end{proof}

Theorem 7 is, of course, the natural continuous state space analog to the Perron-Frobenius theorem that is available in the finite-dimensional setting.

\section*{Acknowledgements}
This work builds upon a research thread that was significantly advanced through the research of Uri Rothblum and Pete Veinott. The first author wishes to take this opportunity to gratefully acknowledge the many years of friendship and intellectual guidance he received from Pete Veinott, starting in his days as a PhD student at Stanford and later deepening when he became a colleague of Pete’s at Stanford. The first author also had the good fortune to collaborate with Uri Rothblum, and to therefore experience firsthand both Pete's and Uri’s creativity and enthusiasm for deep intellectual contribution to our field. They are both greatly missed.

The second author would like to acknowledge academic guidance and education he received from Pete Veinott during his tenure as a PhD student at Stanford. Pete's energy and contributions to our field are very much admired by the second author.

\end{document}